\documentclass[10pt,a4paper]{article}
\usepackage[top=3cm, bottom=3cm, left=3cm, right=3cm]{geometry}
\usepackage[latin1]{inputenc}
\usepackage{amsmath}
\usepackage{amsthm}
\usepackage{amsfonts}
\usepackage{amssymb}
\usepackage{graphics}
\usepackage{tabu}
\usepackage{breqn}
\usepackage[hang,flushmargin]{footmisc} %nonindent footnote

\usepackage{amsmath,amsthm,amssymb,graphicx, multicol, array}
\usepackage{enumerate}
\usepackage{enumitem}
\usepackage{hyperref}
\usepackage{setspace}
\usepackage{xcolor}
\usepackage{currfile}
\usepackage{wrapfig}
\usepackage{tabto}

\usepackage{pgfplots}
\pgfplotsset{width=10cm,compat=1.9}
\newtheorem{thm}{Theorem}[section]
\newtheorem{lem}{Lemma}[section]
\newtheorem{conj}{Conjecture}[section]
\newtheorem{exa}{Example}[section]

\newtheorem{dfn}{Definition}[section]

\newcommand{\N}{\mathbb{N}}
\newcommand{\Z}{\mathbb{Z}}
\newcommand{\Q}{\mathbb{Q}}
\newcommand{\R}{\mathbb{R}}

\usepackage{fancyhdr}
\title{Irrationality Exponents For Even Zeta Constants}
\date{}
\author{N. A. Carella}

\begin{document}
%\doublespacing
\thispagestyle{empty}
\date{}
\maketitle

%\vskip .25 in 

\textbf{\textit{Abstract}:} Let $k\geq 1$ be a small fixed  integer. The rational approximations $\left |p/q-\pi^{k} \right |>1/q^{\mu(\pi^k)}$ of the irrational number $\pi^{k}$ are bounded away from zero. A general result for the irrationality exponent $\mu(\pi^k)$ will be proved here. The specific results and numerical data for a few cases $k=2$ and $k=3$ are also presented and explained. The even parameters $2k$ correspond to the even zeta constants $\zeta(2k)$.\let\thefootnote\relax\footnote{ \today \date{} 
}

%\vskip .25 in 
\tableofcontents
\listoftables
%sssssssssssssssssssssssssssssssssssssssssssssssssssssssssssssssssssssssssssssssssssssssssssssssssssssssssssssssssssssssssssssssssssssssssssssssssssssssssssssssssssssssssssssssssssssssssssssssssssssssssssssssssssssssssssssssssss
\section{Introduction} \label{s6097}
Let $k\geq 1$ be a small fixed  integer. The rational approximations $\left |p/q-\pi^{k} \right |>1/q^{\mu(\pi^k)}$ of the irrational number $\pi^{k}$ are bounded away from zero. The earliest result $\left |p/q-\pi \right |>1/q^{42}$ for the irrationality exponent $\mu(\pi)$ was proved by Mahler in 1953, and more recently it was reduced to $\left |p/q-\pi \right |>1/q^{7.6063}$ by Salikhov in 2008. The earliest result for next number $\left |p/q-\pi^2 \right |>1/q^{11.86}$ was proved by Apery in 1979, and more recently it was reduced to $\left |p/q-\pi^2 \right |>1/q^{7.398537}$ by Rhin and Viola in 1996. Therer is no literature for $k \geq 3$. This note introduces elementary techniques to determine the irrationality exponent $\mu(\pi^k)$ of the irrational number $\pi^k$. It is shown that the Diophantine inequality $\left |p/q-\pi^{k} \right |>1/q^{2+\varepsilon}$, where $\varepsilon>0$ is an arbitrarily small number, is true for any $k\geq1$.

%rrrrrrrrrrrrrrrrrrrrrrrrrrrrrrrrrrrrrrrrrrrrrrrrrrrrrrrrrr
\subsection{Exponent For The Number $\pi^2$}

Let $\{p_n/q_n: n \geq 1\}$ be the sequence of convergents of the irrational number $\pi^2$. The sequence of rational approximations $\{\left |p_n/q_n-\pi^2 \right |: n \geq 1\}$ are bounded away from zero. For instance, the $5$th and $6$th convergents are 
\begin{multicols}{2}
 \begin{enumerate} [font=\normalfont, label=(\roman*)]
\item$ \displaystyle 
 \left |\frac{227}{23}-\pi^2 \right |\geq \frac{1}{23^{3.236253}},
$
\item$\displaystyle
\left |\frac{10748}{1089}-\pi^2\right | \geq \frac{1}{1089^{2.018434}} 
$,
\end{enumerate}
\end{multicols}
respectively, additional data are compiled in Table \ref{t6092.05}. But, it is difficult to prove a lower bound. The earliest result $\left |p/q-\pi^2 \right |\geq 1/q^{11.85}$ was proved by Apery in \cite{AR79}, and more recently it was improved to $\left |p/q-\pi^2 \right |\geq 1/q^{5.44}$ by Rhin and Viola in \cite{RV01}. Basic and elementary ideas are used here to improve it to the followings estimate.
\begin{thm} \label{thm6097.21} For all large rational approximations $p/q \to \pi^2$, the Diophantine inequality
\begin{equation} \label{eq6097.46}
  \left | \pi^2-\frac{p}{q} \right | \geq\frac{1}{q^{2+\varepsilon }},
\end{equation}	
where $\varepsilon>0$ is a small number, is true.
\end{thm}
The proof appears in Section \ref{s6092}.\\

\begin{table}[h!]

\centering
\caption{Historical Data For $\mu(\pi^2)$} \label{t6097.01}
\begin{tabular}{l|l|l}
Irrationality Measure Upper Bound&Reference&Year\\
\hline
 $\mu(\pi^2) \leq 11.85078$ & Apery, \cite{MK53}&1976\\
 $\mu(\pi^2) \leq 10.02979$ & Dvornicich, Viola, \cite{DV87}&1987\\
 $\mu(\pi^2) \leq 5.441243$ & Rhin, Viola, \cite{RV01}&2001\\
\end{tabular}
\end{table}

%rrrrrrrrrrrrrrrrrrrrrrrrrrrrrrrrrrrrrrrrrrrrrrrrrrrrrrrrrr
\subsection{Exponent For The Number $\pi^3$}
Let $\{p_n/q_n: n \geq 1\}$ be the sequence of convergents of the irrational number $\pi^3$. The sequence of rational approximations $\{\left |p_n/q_n-\pi^3 \right |: n \geq 1\}$ are bounded away from zero. For instance, the $5$th and $6$th convergents are 
\begin{multicols}{2}
 \begin{enumerate} [font=\normalfont, label=(\roman*)]
\item$ \displaystyle 
\left |\frac{123498}{3983}-\pi^3 \right |\geq \frac{1}{3983^{2.320380}},
$
\item$\displaystyle
\left |\frac{1714151}{55284}-\pi^3\right | \geq \frac{1}{55284^{2.096515}} 
$,
\end{enumerate}
\end{multicols}
respectively, additional data are compiled Table \ref{t6094.05}. But, it is difficult to prove a lower bound. The literature does not have any estimate nor numerical data on the irrationality exponent of this number. Basic and elementary ideas are used here to prove the followings estimate.
\begin{thm} \label{thm6097.51} For all large rational approximations $p/q \to \pi^3$, the Diophantine inequality
\begin{equation} \label{eq6097.46}
  \left | \pi^3-\frac{p}{q} \right | \geq\frac{1}{q^{2+\varepsilon }},
\end{equation}	
where $\varepsilon>0$ is a small number, is true.
\end{thm}
The proof appears in Section \ref{s6094}.

%rrrrrrrrrrrrrrrrrrrrrrrrrrrrrrrrrrrrrrrrrrrrrrrrrrrrrrrrrr
\subsection{Exponent For The General Case $\pi^k$}
\begin{thm} \label{thm6010.02} Let $k\geq 1$ be a small fixed integer. For all large rational approximations $p/q \to \pi^k$, the Diophantine inequality
\begin{equation} \label{eq6010.06}
  \left | \pi^k-\frac{p}{q} \right | \gg\frac{1}{q^{2+\varepsilon }},
\end{equation}	
where $\varepsilon>0$ is a small number, is true.
\end{thm}
The proof appears in Section \ref{s7729}.\\

%sssssssssssssssssssssssssssssssssssssssssssssssssssssssssssssssssssssssssssssssssssssssssssssssssssssssssssssssssssssssssssssssssssssssssssssssssssssssssssssssssssssssssssssssssssssssssssssssssssssssssssssssssssssssssssssssssss
\section{Harmonic Summation Kernels}\label{s5534}
The harmonic summation kernels naturally arise in the partial sums of Fourier series and in the studies of convergences of continuous functions.

\begin{dfn} \label{dfn6634.100} The Dirichlet kernel is defined by
\begin{equation} \label{eq5534.200}
\mathcal{D}_x(z)=\sum_{-x\leq n \leq x} e^{i 2nz}=\frac{\sin((2x+1)z)}{ \sin \left ( z \right )},
\end{equation} 
where $x\in \N$ is an integer and $ z \in \R-\pi\Z$ is a real number. 
\end{dfn}

\begin{dfn} \label{dfn6634.102} The Fejer kernel is defined by
\begin{equation} \label{eq5534.204}
\mathcal{F}_x(z)=\sum_{0\leq n \leq x,} \sum_{-n\leq k \leq n} e^{i 2kz}=\frac{1}{2}\frac{\sin((x+1)z)^2}{ \sin \left ( z \right )^2},
\end{equation} 
where $x\in \N$ is an integer and $ z \in \R-\pi\Z$ is a real number. 
\end{dfn}

These formulas are well known, see \cite{KT89} and similar references. For $z \ne k \pi$, the harmonic summation kernels have the upper bounds $\left |\mathcal{K}_x(z) \right |=\left |\mathcal{D}_x(z) \right | \ll |x|$, and $\left |\mathcal{K}_x(z) \right |=\left |\mathcal{F}_x(z) \right | \ll |x^2|$. \\

The Dirichlet kernel in Definition \ref{dfn6634.100} is a well defined continued function of two variables $x,z \in \R$. Hence, for fixed $z$, it has an analytic continuation to all the real numbers $x \in \R$. \\

An important property is the that a proper choice of the parameter $x\geq1$ can shifts the sporadic large value of the reciprocal sine function $1/\sin z$ to $\mathcal{K}_x(z)$, and the term $1/\sin(2x+1)z$ remains bounded. This principle will be applied to certain lacunary sequences $\{q_n : n \geq 1\}$, which maximize the reciprocal sine function $1/\sin z$, to obtain an effective upper bound of the function $1/\sin z$.\\

%lllllllllllllllllllllllllllllllllllllllll
\begin{lem}\label{lem5534.505} Let $k\geq 1$ be a small fixed integer, and let $\{p_n/q_n : n \geq 1\}$ be the sequence of convergents of the real number $\pi^{k}$, and $0\ne z \in \Z$. Then
\begin{equation} \label{eq5534.520}
\frac{1}{\left |\sin(\pi^{k+1}z)\right |}\ll \frac{1}{\left |\sin\left (\pi^{k+1}q_n\right )\right |}.
\end{equation} 
\end{lem}
\begin{proof} By the best approximation principle, see Lemma \ref{lem2000.07}, 
\begin{equation} \label{eq5534.522}
\left | m-\pi^{k} z\right |\geq \left | p_n-\pi^{k} q_n\right |
\end{equation} 
for any integer $z \leq q_n$. Hence, 
\begin{eqnarray} \label{eq5534.573}
\frac{1}{\left | \sin\left ( \pi^{k+1} z\right) \right |}
&=&\frac{1}{\left |\sin\left ( \pi m-\pi^{k+1} z\right)\right |} \\
&\leq&\frac{1}{\left |\sin\left ( \pi p_n-\pi^{k+1} q_n\right)\right |} \nonumber\\
&=& \frac{1}{\left |\sin\left ( \pi^{k+1} q_n \right)\right | }\nonumber,
\end{eqnarray}
as $n \to \infty$. 
\end{proof}

%sssssssssssssssssssssssssssssssssssssssssssssssssssssssssssssssssssssssssssssssssssssssssssssssssssssssssssssssssssssssssssssssssssssssssssssssssssssssssssssssssssssssssssssssssssssssssssssssssssssssssssssssssssssssssssssssssss
\section{Upper Bound For $\left | 1/\sin \pi^{k+1} z \right |$} \label{s6699}
As shown in Lemma \ref{lem5534.505}, to estimate the upper bound of the function $1/|\sin \pi^{k+1} z|$ over the real numbers $z \in\R$, it is sufficient to fix $z=q_n$, and select a real number $x \in \R$ such that $q_n \asymp x$. This idea is demonstrated below for small integer parameter $k\geq 1$.

%lllllllllllllllllllllllllllllllllllllllll
\begin{lem}\label{lem6699.705} Let $k\geq 1$ be a small fixed integer, let $\{p_n/q_n : n \geq 1\}$ be the sequence of convergents of the real number $\pi^{k}$, and define the associated sequence
\begin{equation} \label{eq6699.702} 
 x_n=\left (\frac{2^{2+2v_2}+1}{2^{2+2v_2}}\right )\frac{q_n}{\pi^{k}},
\end{equation} 
where $v_2=v_2(q_n)=\max\{v:2^v\mid q_n\}$ is the $2$-adic valuation, and $n \geq 1$. Then
 \begin{enumerate} [font=\normalfont, label=(\roman*)]
\item$\displaystyle  \sin\left (2 (x_n-1/2)+1)\pi^{k+1}q_n\right ) =\pm 1$.
\item$\displaystyle  \sin\left (2 (x_n+1/2)+1)\pi^{k+1}q_n\right ) =\pm \cos 2\pi^{k+1}q_n$.
\item$\displaystyle  \left | \sin\left (2 x_n+1/2)\pi^{k+1}q_n\right ) \right |\geq 1 -\frac{2\pi^2}{q_n^2},$  \tabto{7cm} as $n \to \infty$.
\end{enumerate}
\end{lem}
\begin{proof} Observe that the value $x_n$ in \eqref{eq6699.702} yields 
\begin{equation} \label{eq6699.720}
\sin(2\pi^{k+1}q_nx_n)=\sin\left (2\pi^{k+1}q_n\left (\frac{2^{2+2v_2}+1}{2^{2+2v_2}}\right )\frac{q_n}{\pi^{k}}\right )=\sin\left (\frac{\pi}{2}\cdot w_n\right )=\pm1,
\end{equation}
and 
\begin{equation} \label{eq6699.724}
\cos\left (2\pi^{k+1}q_nx_n\right )=\cos\left (2\pi^{k+1}q_n\left (\frac{2^{2+2v_2}+1}{2^{2+2v_2}}\right )\frac{q_n}{\pi^{k}}\right )=\cos\left (\frac{\pi}{2}\cdot w_n\right )=0,
\end{equation}
where
\begin{equation} \label{eq6699.722}
w_n=\left (\frac{2^{2+2v_2}+1}{2^{2v_2}}\right ) q_n^{2} 
\end{equation}
is an odd integer.
(i) Routine calculations yield this:
\begin{eqnarray} \label{eq6699.703}
\sin((2(x_n-1/2)+1) \pi^{k+1}q_n)
&=&\sin\left (2\pi^{k+1}q_nx_n\right ) \\
&=&\sin\left (2\pi^{k+1}q_n\left (\frac{2^{2+2v_2}+1}{2^{2+2v_2}}\right )\frac{q_n}{\pi^{k}}\right )  \nonumber\\
&=&\sin\left (\frac{\pi}{2}\cdot w_n\right )  \nonumber\\
&=&\pm1 \nonumber,
\end{eqnarray}
(ii) Routine calculations yield this:
\begin{eqnarray} \label{eq6699.703}
\sin\left ((2(x_n+1/2) +1) \pi^{k+1}q_n \right )
&=&\sin(2\pi^{k+1}q_nx_n+ 2\pi^{k+1}q_n) \\
&=&\sin(2\pi^{k+1}q_nx_n)\cos( 2\pi^{k+1}q_n) \nonumber\\
&&\qquad \qquad + \cos(2\pi^{k+1}q_nx_n)\sin( 2\pi^{k+1}q_n)  \nonumber.
\end{eqnarray}
Substituting \eqref{eq6699.720} and \eqref{eq6699.724} into \eqref{eq6699.703} return 
\begin{equation} \label{eq6699.725}
\sin\left (2(x_n+1/2)+1)\pi^{k+1}q_n\right )=\pm\cos\left (2 \pi^{k+1}q_n\right).
\end{equation}
(iii) This follows from the previous result:
\begin{eqnarray} \label{eq6699.727}
\left |\sin\left (2(x_n+1/2)+1)\pi^{k+1}q_n\right ) \right |
&=&\left |\pm\cos\left (2 \pi^{k+1}q_n\right) \right |\\
&=&\left |\pm\cos\left ( 2\pi p_n-2\pi^{k+1} q_n\right)\right | \nonumber\\
&=&\left |\pm\cos\left ( 2\pi \left 
(p_n-\pi^{k} q_n\right ) \right)\right | \nonumber\\
&\asymp  &1\nonumber,
\end{eqnarray}
since the sequence of convergents satisfies $\left | p_n-\pi^{k}q_n \right | \leq 1/q_n$ as $n \to \infty$. 
\end{proof}
%llllllllllllllllllllllllllllllllllllllllllllllllllll
\begin{lem}\label{lem6699.805} Let $k\geq 1$ be a small fixed integer, let $\{p_n/q_n : n \geq 1\}$ be the sequence of convergents of the real number $\pi^{k}$, and define the associated sequence
\begin{equation} \label{eq6699.802} 
 x_n=\left (\frac{2^{2+2v_2}+1}{2^{2+2v_2}}\right )\frac{q_n}{\pi^{k}}  ,
\end{equation} 
where $v_2=v_2(q_n)=\max\{v:2^v\mid q_n\}$ is the $2$-adic valuation, and $n \geq 1$. Then
\begin{equation} \label{eq6699.725}
\left |\sin\left (2x^{*}+1)\pi^{k+1}q_n\right ) \right |\asymp 1,
\end{equation}
where $x^{*}\in [x_n-1/2, x_n+1/2]$ is an integer.
\end{lem}

\begin{proof} Consider the continuous function $f(x)=\left |\sin\left (2x+1)\pi^{k+1}q_n\right ) \right |$ over the interval $[x_n-1/2, x_n+1/2]$. By Lemma \ref{lem6699.705}, it has a local maximal at $x=x_n-1/2 \in \R$: 
\begin{eqnarray} \label{eq6699.303}
\left | \sin\left ((2x+1)\pi^{k+1}z\right )\right | &= &\left |\sin\left ((2(x_n-1/2)+1) \pi^{k+1} q_n\right )\right |\\
&=&1 \nonumber,
\end{eqnarray}
and it has a local minimal at $x=x_n+1/2 \in \R$:
\begin{eqnarray} \label{eq6699.303}
\left | \sin\left ((2x+1)\pi^{k+1}z\right )\right | &= &\left |\sin\left ((2(x_n+1/2)+1) \pi^{k+1} q_n\right )\right |\\
&\geq &1 -\frac{2\pi^2}{q_n^2} \nonumber.
\end{eqnarray}
Since $f(x)$ is continuous over the interval $[x_n-1/2, x_n+1/2]$, it follows that  
\begin{equation} \label{eq6699.877}
1 -\frac{2\pi^2}{q_n^2} \leq \left | \sin\left ((2x^{*}+1)\pi^{k+1}z\right )\right | \leq 1 \nonumber
\end{equation}
for any integer $x^{*}\in [x_n-1/2, x_n+1/2]$ \end{proof}

%tttttttttttttttttttttttttttttttttttttttttttttttttttttttttt
\begin{thm} \label{thm6699.300}  If $k\geq 1$ is a small fixed integer, and $z \in \N$ is a large integer, then,
\begin{equation} \label{eq6699.300}
 \left |\frac{1}{\sin \pi^{k+1} z}\right |\ll \left |z\right |.
\end{equation}
\end{thm}
\begin{proof} Let $\{p_n/q_n : n \geq 1\}$ be the sequence of convergents of the real number $\pi^{k}$. Since the denominators sequence $\{q_n : n \geq 1\}$ maximize the reciprocal sine function $1/\sin \pi^{k+1} z$, see Lemma \ref{lem5534.505}, it is sufficient to prove it for $z=q_n$. Define the associated sequence
\begin{equation} \label{eq6699.302} 
 x_n=\left (\frac{2^{2+2v_2}+1}{2^{2+2v_2}}\right )\frac{q_n}{\pi^{k}}  ,
\end{equation} 
where $v_2=v_2(q_n)=\max\{v:2^v\mid q_n\}$ is the $2$-adic valuation, and $n \geq 1$. Let $f(x)=\left | \sin\left ((2x+1)\pi^{k+1}z\right )\right | $, and let $z= q_n$. The function $f(x)$ is bounded over the interval $[x_n-1/2,x_n+1/2]$, see Lemma \ref{lem6699.705}.  Replacing the integer parameters $x^{*}\in [x_n-1/2, x_n+1/2]$, $z= q_n$, and applying Lemma \ref{lem6699.705} return 
\begin{eqnarray} \label{eq6699.303}
\left | \sin\left ((2x+1)\pi^{k+1}z\right )\right | &= &\left |\sin\left ((2x^{*}+1) \pi^{k+1} q_n\right )\right |\\
&\asymp&1 \nonumber.
\end{eqnarray}
Rewrite the reciprocal sine function in terms of the harmonic kernel in Definition \ref{dfn6634.100}, and splice all these information together, to obtain
\begin{eqnarray} \label{eq6699.313}
 \left |\frac{1}{\sin \pi^{k+1}z}\right |  &=& \left |\frac{\mathcal{D}_{x}(\pi^{k+1}z)}{\sin((2x+1) \pi^{k+1}z)}\right |\nonumber\\
&\ll&\left |\mathcal{D}_{x^{*}}  \right | \left |\frac{1}{\sin((2x^{*}+1) \pi^{k+1}q_n)}\right |\\
&\ll&\left |x^{*}\right |\cdot 1\nonumber\\
&\ll&\left |z\right |\nonumber
\end{eqnarray}
since $|z|\asymp x^{*}\asymp p_n\asymp q_n$, and the trivial estimate $\left |\mathcal{D}_x(z) \right | \ll \left |x\right |$.
\end{proof}
%sssssssssssssssssssssssssssssssssssssssssssssssssssssssssssssssssssssssssssssssssssssssssssssssssssssssssssssssssssssssssssssssssssssssssssssssssssssssssssssssssssssssssssssssssssssssssssssssssssssssssssssssssssssssssssssssssss
\section{Upper Bound For $\left | 1/\sin \pi^3 z \right |$}\label{s6634}
As shown in Lemma \ref{lem5534.505}, to estimate the upper bound of the function $1/|\sin \pi^{3} z|$ over the real numbers $z \in\R$, it is sufficient to fix $z=q_n$, and select a real number $x \in \R$ such that $q_n \asymp x$. This idea is demonstrated below.

\begin{lem}\label{lem6634.705} Let $\{p_n/q_n : n \geq 1\}$ be the sequence of convergents of the real number $\pi^{2}$, and define the associated sequence
\begin{equation} \label{eq6634.702} 
 x_n=\left (\frac{2^{2+2v_2}+1}{2^{2+2v_2}}\right )\frac{q_n}{\pi^{2}}  ,
\end{equation} 
where $v_2=v_2(q_n)=\max\{v:2^v\mid q_n\}$ is the $2$-adic valuation, and $n \geq 1$. Then
 \begin{enumerate} [font=\normalfont, label=(\roman*)]
\item$\displaystyle  \sin\left (2 (x_n-1/2)+1)\pi^{3}q_n\right ) =\pm 1$.
\item$\displaystyle  \sin\left (2 (x_n+1/2)+1)\pi^{3}q_n\right ) =\pm \cos 2\pi^{3}q_n$.
\item$\displaystyle  \left | \sin\left (2 x_n+1/2)\pi^{3}q_n\right ) \right |\geq 1 -\frac{2\pi^2}{q_n^2},$  as $x \to \infty$.
\end{enumerate}
\end{lem}
\begin{proof}
Same as Lemma \ref{lem6699.705}.
\end{proof}
\begin{lem}\label{lem6634.805} Let $\{p_n/q_n : n \geq 1\}$ be the sequence of convergents of the real number $\pi^{2}$, and define the associated sequence
\begin{equation} \label{eq6634.802} 
 x_n=\left (\frac{2^{2+2v_2}+1}{2^{2+2v_2}}\right )\frac{q_n}{\pi^{k}}  ,
\end{equation} 
where $v_2=v_2(q_n)=\max\{v:2^v\mid q_n\}$ is the $2$-adic valuation, and $n \geq 1$. Then
\begin{equation} \label{eq6634.725}
\left |\sin\left (2x^{*}+1)\pi^{3}q_n\right ) \right |\asymp 1,
\end{equation}
where $x^{*}\in [x_n-1/2, x_n+1/2]$ is an integer.
\end{lem}
\begin{proof}
Same as Lemma \ref{lem6699.805}.
\end{proof}
\begin{thm} \label{thm6634.300}   Let $z \in \N$ be a large integer. Then,
\begin{equation} \label{eq6634.300}
 \left |\frac{1}{\sin \pi^3 z}\right |\ll \left |z\right |.
\end{equation}
\end{thm}
\begin{proof} Let $\{p_n/q_n : n \geq 1\}$ be the sequence of convergents of the real number $\pi^2$. Since the denominators sequence $\{q_n : n \geq 1\}$ maximize the reciprocal sine function $1/\sin \pi^3 z$, it is sufficient to prove it for $z=q_n$. Define the associated sequence
\begin{equation} \label{eq6634.302} 
 x_n=\left (\frac{2^{2+2v_2}+1}{2^{2+2v_2}}\right )\frac{q_n}{\pi^2}  ,
\end{equation} 
where $v_2=v_2(q_n)=\max\{v:2^v\mid q_n\}$ is the $2$-adic valuation, and $n \geq 1$. Replacing the integer parameters $x^{*}\in [x_n-1/2, x_n+1/2]$, $z= q_n$, and applying Lemma \ref{lem6634.805} return 
\begin{eqnarray} \label{eq6634.303}
\left | \sin\left ((2x+1)\pi^3z\right )\right | &= &\left |\sin\left ((2x^{*}+1) \pi^3 q_n\right )\right |\\
&\asymp&1 \nonumber,
\end{eqnarray}
since the sequence of convergents satisfies $\left | p_n-\pi^2q_n \right | \to 0 $ as $n \to \infty$. Rewrite the reciprocal sine function in terms of the harmonic kernel in Definition \ref{dfn6634.100}, and splice all these information together, to obtain
\begin{eqnarray} \label{eq6634.313}
 \left |\frac{1}{\sin \pi^3z}\right |  &=& \left |\frac{\mathcal{D}_{x}(\pi^3z)}{\sin((2x+1) \pi^3z)}\right |\nonumber\\
&\ll&\left |\mathcal{D}_{x^{*}}  \right | \left |\frac{1}{\sin((2x^{*}+1) \pi^3q_n)}\right |\\
&\ll&\left |x^{*}\right |\cdot 1\nonumber\\
&\ll&\left |z\right |\nonumber
\end{eqnarray}
since $|z|\asymp x^{*}\asymp p_n\asymp q_n$, and the trivial estimate $\left |\mathcal{D}_x(z) \right | \ll \left |x\right |$.
\end{proof}

%sssssssssssssssssssssssssssssssssssssssssssssssssssssssssssssssssssssssssssssssssssssssssssssssssssssssssssssssssssssssssssssssssssssssssssssssssssssssssssssssssssssssssssssssssssssssssssssssssssssssssssssssssssssssssssssssssss
\section{Upper Bound For $1/\left | \sin \pi^4 z \right |$}\label{s7734}
As shown in Lemma \ref{lem5534.505}, to estimate the upper bound of the function $1/|\sin \pi^{4} z|$ over the real numbers $z \in\R$, it is sufficient to fix $z=q_n$, and select a real number $x \in \R$ such that $q_n \asymp x$. This idea is demonstrated below.

\begin{lem}\label{lem7734.705} Let $\{p_n/q_n : n \geq 1\}$ be the sequence of convergents of the real number $\pi^{3}$, and define the associated sequence
\begin{equation} \label{eq7734.702} 
 x_n=\left (\frac{2^{2+2v_2}+1}{2^{2+2v_2}}\right )\frac{q_n}{\pi^{3}}  ,
\end{equation} 
where $v_2=v_2(q_n)=\max\{v:2^v\mid q_n\}$ is the $2$-adic valuation, and $n \geq 1$. Then
 \begin{enumerate} [font=\normalfont, label=(\roman*)]
\item$\displaystyle  \sin\left (2 (x_n-1/2)+1)\pi^{4}q_n\right ) =\pm 1$.
\item$\displaystyle  \sin\left (2 (x_n+1/2)+1)\pi^{4}q_n\right ) =\pm \cos 2\pi^{3}q_n$.
\item$\displaystyle  \left | \sin\left (2 x_n+1/2)\pi^{4}q_n\right ) \right |\geq 1 -\frac{2\pi^2}{q_n^2},$  as $x \to \infty$.
\end{enumerate}
\end{lem}
\begin{proof}
Same as Lemma \ref{lem6699.705}.
\end{proof}
\begin{lem}\label{lem7734.805} Let $\{p_n/q_n : n \geq 1\}$ be the sequence of convergents of the real number $\pi^{3}$, and define the associated sequence
\begin{equation} \label{eq7734.802} 
 x_n=\left (\frac{2^{2+2v_2}+1}{2^{2+2v_2}}\right )\frac{q_n}{\pi^{3}}  ,
\end{equation} 
where $v_2=v_2(q_n)=\max\{v:2^v\mid q_n\}$ is the $2$-adic valuation, and $n \geq 1$. Then
\begin{equation} \label{eq7734.725}
\left |\sin\left (2x^{*}+1)\pi^{4}q_n\right ) \right |\asymp 1,
\end{equation}
where $x^{*}\in [x_n-1/2, x_n+1/2]$ is an integer.
\end{lem}
\begin{proof}
Same as Lemma \ref{lem6699.805}.
\end{proof}
\begin{thm} \label{thm7734.304}   Let $z \in \N$ be a large integer. Then,
\begin{equation} \label{eq7734.300}
 \left |\frac{1}{\sin \pi^4 z}\right |\ll \left |z\right |.
\end{equation}
\end{thm}
\begin{proof} Let $\{p_n/q_n : n \geq 1\}$ be the sequence of convergents of the real number $\pi^3$. Since the denominators sequence $\{q_n : n \geq 1\}$ maximize the reciprocal sine function $1/\sin \pi^4 z$, it is sufficient to prove it for $z=q_n$. Define the associated sequence
\begin{equation} \label{eq7734.302} 
 x_n=\left (\frac{2^{2+2v_2}+1}{2^{2+2v_2}}\right )\frac{q_n}{\pi^3}  ,
\end{equation} 
where $v_2=v_2(q_n)=\max\{v:2^v\mid q_n\}$ is the $2$-adic valuation, and $n \geq 1$. Replacing the integer parameters $x^{*}\in [x_n-1/2, x_n+1/2]$, $z= q_n$, and applying Lemma \ref{lem7734.805} return 
\begin{eqnarray} \label{eq7734.303}
\left | \sin\left ((2x+1)z\right )\right | &= &\left |\sin\left ((2x^{*}+1) \pi^4 q_n\right )\right |\\
&\asymp&1 \nonumber,
\end{eqnarray}
since the sequence of convergents satisfies $\left | p_n-\pi^3q_n \right | \to 0 $ as $n \to \infty$. Rewrite the reciprocal sine function in terms of the harmonic kernel in Definition \ref{dfn6634.100}, and splice all these information together, to obtain
\begin{eqnarray} \label{eq7734.313}
 \left |\frac{1}{\sin z}\right |  &=& \left |\frac{\mathcal{D}_{x}(z)}{\sin((2x+1) z)}\right |\nonumber\\
&\ll&\left |\mathcal{D}_{x^{*}}  \right | \left |\frac{1}{\sin((2x^{*}+1) \pi^4 q_n)}\right |\\
&\ll&\left |x^{*}\right |\cdot 1\nonumber\\
&\ll&\left |z\right |\nonumber
\end{eqnarray}
since $|z|\asymp x^{*}\asymp p_n\asymp q_n$, and the trivial estimate $\left |\mathcal{D}_x(z) \right | \ll \left |x\right |$.
\end{proof}

%sssssssssssssssssssssssssssssssssssssssssssssssssssssssssssssssssssssssssssssssssssssssssssssssssssssssssssssssssssssssssssssssssssssssssssssssssssssssssssssssssssssssssssssssssssssssssssssssssssssssssssssssssssssssssssssssssss
\section{The Exponent Result For $\pi^2$ } \label{s6092}
The last estimate for irrationality exponent of the first even zeta constant $\zeta(2)=\pi^2/6$ in Table \ref{t6097.01} was derived from the algebraic properties of the cellular integral
\begin{equation} \label{eq6092.100}
A_2+B_2\zeta(2)=  \int_0^1\int_0^1 \frac{x^h(1-x)^iy^j(1-y)^k}{\left( 1-xy\right)^{i+j-l}}\frac{dx\, dy }{1-xy },
\end{equation} 
where $A_2,B_2 \in \Z$ are integers. The analysis appears in \cite{RV96}, and an expanded version of the theory of cellular integrals is presented in \cite[Section 5.3]{BF14}. These techniques also rely on rational functions approximations of $\pi^2$ and the prime number theorem. Some relevant references are \cite{RV96},  \cite{HM00}, \cite{HM93}, \cite{SV08}, and \cite{BF00} for an introduction to the rational approximations of $\pi$ and the various proofs.\\

Since $\zeta(2)$ and $\pi^2$ have the same irrationality exponent, the analysis is done for the simpler number. The proof within is based on an effective upper bound of the reciprocal sine function over the sequence $\{q_n: n\geq 1\}$ as derived in Section \ref{s6634}. 

\begin{proof} (Theorem \ref{thm6097.21}) Let $\varepsilon >0$ be an arbitrary small number, and let $\{p_n/q_n: n \geq 1\}$ be the sequence of convergents of the irrational number $\pi^2$. By Theorem \ref{thm6634.300}, the reciprocal sine function has the upper bound
\begin{equation} \label{eq6092.883}
 \left |\frac{1}{\sin \left ( \pi^3q_n \right ) }\right |\ll q_n^{1+\varepsilon} .
\end{equation} 
Moreover, $\sin \left ( \pi^3q_n \right )= \sin \left (\alpha p -\pi^3 q_n \right )$ if and only if $\alpha p=\pi p_n$, where $p_n$ is an integer. These information lead to the following relation.  
\begin{eqnarray} \label{eq6092.885}
\frac{1}{q_n^{1+\varepsilon}}&\ll&  \left |\sin \left (\pi^3q_n  \right )\right |\\
&\ll& \left |\sin \left (\pi^3q_n -\pi p_n \right )\right |\nonumber\\
&\ll& \left |\sin \left (\pi \left ( \pi^2q_n - p_n \right )\right )\right |\nonumber\\
&\ll& \left |\pi^2 q_n-p_n  \right |\nonumber
\end{eqnarray}
for all sufficiently large $p_n/q_n$. Therefore, 
\begin{eqnarray} \label{eq6092.36}
  \left | \pi^2-\frac{p_n}{q_n} \right | 
&\gg&\frac{1}{q^{2+\varepsilon} }\\
&=&\frac{1}{q^{\mu(\pi^2)+\varepsilon} }\nonumber.
\end{eqnarray}
Clearly, this implies that the irrationality measure of the real number $\pi^2$ is $\mu(\pi^2)=2$, see Definition \ref{dfn2000.01}. 
Quod erat demontrandum.
\end{proof}
%tttttttttttttttttttttttttttttttttttttttttttttttttttttttttttttttttttttttttttttttttttttttttttttttttttttttttttttttttttttttttttttttttttttttttttttttttttttttttttttttttttt
%AAAAAAAAAAAAAAAAAAAAAAAAAAAAAAAAAAAAAAAAAAAAAAAAAAAAAAAAAAAAAAAAAAAAAAAAAAAAAAA
\subsection{Numerical Data For The Exponent $\mu(\pi^2)$}
The continued fraction is
\begin{equation}\label{eq6092.110}
\pi^2=[9; 1, 6, 1, 2, 47, 1, 8, 1, 1, 2, 2, 1, 1, 8, 3, 1, 10, 5, 1, 3, 1, 2, 1, 1, 3, 15, \ldots ].   
\end{equation}
The sequence of convergents $\{p_n/q_n: n \geq 1\}$ is computed via the recursive formula provided in Lemma \ref{lem2000.101}. The approximation $\mu_n(\pi^2)$ of the exponent in the inequality
\begin{equation} \label{eq6092.46}
  \left | \pi^2-\frac{p_n}{q_n} \right | 
\geq\frac{1}{q^{\mu_n(\pi^2) }}
\end{equation}
are tabulated in Table \ref{t6092.05} for the early stage of the sequence of convergents $p_n/q_n \longrightarrow \pi^2$. 
%\newpage
%\begin{wraptable}
\begin{table}[h!]
\centering
\caption{Numerical Data For The Exponent $\mu(\pi^2)$} \label{t6092.05}
\begin{tabular}{l|l|l| l}
$n$&$p_n$&$q_n$&$\mu_n(\pi^2)$\\
\hline
1&$9$&   $1$   &$ $\\
2&$10$&  $1$   &$ $\\
3&$69$&   $7$   &$2.253500$\\
4&$79$&  $8$   &$2.511334$\\
5&$227$&   $23$   &$3.236253$\\
6&$10748$&  $1089$   &$2.018434$\\
7&$10975$&  $1112$   &$2.321958$\\
8&$98548$&   $9985$   &$2.064841$\\
9&$109523$&  $11097$   &$2.090224$\\
10&$208071$&  $21082$   &$2.107694$\\ 
11&$525665$&      $53261$   &$2.098602$ \\ 
12&$1259401$&   $127604$   &$2.071191$\\
13&$1785066$&   $180865$   &$2.049770$\\
14&$3044467$& $308469$ & $2.172439$ \\
15&$26140802$&  $2648617$ & $2.094189$ \\ 
16&$81466873$&  $8254320$ & $2.021982$ \\
17&$107607675$&    $10902937$ & $2.147582$ \\
18&$1157543623$&  $117283690$ & $2.095357$ \\
19&$5895325790$&  $597321387$ & $2.018903$ \\
20&$7052869413$&  $714605077$ & $2.074380$ \\
21&$27053934029$&  $2741136618$ & $2.023038$ \\
22&$34106803442$&  $3455741695$ & $2.055226$ \\
23&$95267540913$&  $9652620008$ & $2.032519$ \\
24&$129374344355$&$13108361703$&$2.031079$\\
25&$224641885268$&  $22760981711$& $2.054176$\\
26&$803300000159$&   $81391306836$   &$2.110031 $\\
27&$12274141887653$&  $1243630584251$   &$2.020459 $\\
28&$13077441887812$&   $1325021891087$   &$2.030798$\\
29&$25351583775465$&  $2568652475338$   &$2.036971$\\
30&$63780609438742$&  $6462326841763$ & $2.039154$
\end{tabular}
\end{table}
%\end{wraptable}

%sssssssssssssssssssssssssssssssssssssssssssssssssssssssssssssssssssssssssssssssssssssssssssssssssssssssssssssssssssssssssssssssssssssssssssssssssssssssssssssssssssssssssssssssssssssssssssssssssssssssssssssssssssssssssssssssssss
\section{The Exponent Result For $\pi^3$ } \label{s6094}
The literature seems to offer no information on the irrationality exponent $\mu(\pi^3)\geq 2$ of the irrational number $\pi^3$. 

\begin{proof} (Theorem \ref{thm6097.51}) Let $\varepsilon >0$ be an arbitrary small number, and let $\{p_n/q_n: n \geq 1\}$ be the sequence of convergents of the irrational number $\pi^3$. By Theorem \ref{thm7734.304}, the reciprocal sine function has the upper bound
\begin{equation} \label{eq6094.883}
 \left |\frac{1}{\sin \left ( \pi^4q_n \right ) }\right |\ll q_n^{1+\varepsilon} .
\end{equation} 
Moreover, $\sin \left ( \pi^4q_n \right )= \sin \left (\alpha p -\pi^4 q_n \right )$ if and only if $\alpha p=\pi p_n$, where $p_n$ is an integer. These information lead to the following relation.  
\begin{eqnarray} \label{eq6094.885}
\frac{1}{q_n^{1+\varepsilon}}&\ll&  \left |\sin \left (\pi^4q_n  \right )\right |\\
&\ll& \left |\sin \left (\pi^4q_n -\pi p_n \right )\right |\nonumber\\
&\ll& \left |\sin \left (\pi \left | \pi^3q_n - p_n \right |\right )\right |\nonumber\\
&\ll& \left |\pi^3 q_n- p_n \right |\nonumber
\end{eqnarray}
for all sufficiently large $p_n/q_n$. Therefore, 
\begin{eqnarray} \label{eq6094.36}
  \left | \pi^3-\frac{p_n}{q_n} \right | 
&\gg&\frac{1}{q^{2+\varepsilon} }\\
&=&\frac{1}{q^{\mu(\pi^3)+\varepsilon} }\nonumber.
\end{eqnarray}
Clearly, this implies that the irrationality measure of the real number $\pi^3$ is $\mu(\pi^3)=2$, see Definition \ref{dfn2000.01}. 
Quod erat demontrandum.
\end{proof}

%tttttttttttttttttttttttttttttttttttttttttttttttttttttttttttttttttttttttttttttttttttttttttttttttttttttttttttttttttttttttttttttttttttttttttttttttttttttttttttttttttttt
%AAAAAAAAAAAAAAAAAAAAAAAAAAAAAAAAAAAAAAAAAAAAAAAAAAAAAAAAAAAAAAAAAAAAAAAAAAAAAAA
\subsection{Numerical Data For The Exponent $\mu(\pi^3)$}
The continued fraction of the second odd power of $\pi$ is
\begin{equation}\label{eq6094.110}
\pi^3=[31; 159, 3, 7, 1, 13, 2, 1, 3, 1, 12, 2, 2, 4, 34, 2, 43, 3, 1, 3, 2, ...  \ldots ].   
\end{equation}
The sequence of convergents $\{p_n/q_n: n \geq 1\}$ is computed via the recursive formula provided in Lemma \ref{lem2000.101}.
The approximation $\mu_n(\pi^3)$ of the exponent in the inequality
\begin{equation} \label{eq6094.46}
  \left | \pi^3-\frac{p_n}{q_n} \right | 
\geq\frac{1}{q^{\mu_n(\pi^3) }}
\end{equation}
are tabulated in Table \ref{t6094.05} for the early stage of the sequence of convergents $p_n/q_n \longrightarrow \pi^3$. 
%\newpage
\begin{table}[h!]
\centering
\caption{Numerical Data For The Exponent $\mu(\pi^3)$} \label{t6094.05}
\begin{tabular}{l|l|l| l}
$n$&$p_n$&$q_n$&$\mu_n(\pi^3)$\\
\hline
1&$31$&   $1$   &$ $\\
2&$4930$&  $159$   &$2.225255 $\\
3&$14821$&   $478$   &$2.342289$\\
4&$108677$&  $3505$   &$2.023480$\\
5&$123498$&   $3983$   &$2.320380$\\
6&$1714151$&  $55284$   &$2.096515$\\
7&$3551800$&  $114551$   &$2.047419$\\
8&$5265951$&   $169835$   &$2.126720$\\
9&$19349653$&  $624056$   &$2.022641$\\
10&$24615604$&  $793891$   &$2.189908$\\ 
11&$314736901$&      $10150748$   &$2.057364$ \\ 
12&$654089406$&   $21095387$   &$2.059538$\\
13&$1622915713$&   $52341522$   &$2.083769$\\
14&$7145752258$& $230461475$ & $2.184225$ \\
15&$244578492485$&  $7888031672$ & $2.031550$ \\ 
16&$496302737228$&  $16006524819$ & $2.160820$ \\
17&$21585596193289$&    $696168598889$ & $2.048912$ \\
18&$65253091317095$&  $2104512321486$ & $2.017121$ \\
19&$86838687510384$&  $2800680920375$ & $2.049611$ \\
20&$325769153848247$&  $10506555082611$ & $2.034434$ \\
21&$738376995206878$&  $23813791085597$ & $2.026878$ \\
22&$1064146149055125$&  $34320346168208$ & $2.020155$ \\
23&$1802523144262003$&  $58134137253805$ & $2.057247$ \\
24&$10076761870365140$&$324991032437233$&$2.020858$\\
25&$11879285014627143$&  $383125169691038$& $2.021449$\\
26&$21956046884992283$&   $708116202128271$   &$2.049213 $\\
27&$99703472554596275$&  $3215589978204122$   &$2.009654 $\\
28&$121659519439588558$&   $3923706180332393$   &$2.050107$\\
29&$708001069752539065$&  $22834120879866087$   &$2.040614$\\
30&$2953663798449744818$&  $95260189699796741$ & $2.023276$\\
\end{tabular}
\end{table}

%sssssssssssssssssssssssssssssssssssssssssssssssssssssssssssssssssssssssssssssssssssssssssssssssssssssssssssssssssssssssssssssssssssssssssssssssssssssssssssssssssssssssssssssssssssssssssssssssssssssssssssssssssssssssssssssssssss
\section{The Exponent Result For The Odd $\zeta(3)$ } \label{s6093}
The last estimate for irrationality exponent of the odd zeta constant $\zeta(3)$ was derived from the algebraic properties of the cellular integral
\begin{equation} \label{eq6093.100}
A_3+B_3\zeta(3)= \int_0^1 \int_0^1\int_0^1 \frac{x^h(1-x)^ly^sz^j(1-z)^q}{\left( 1-(1-xy)z\right)^{q+h-r}}\frac{dx\, dy \,dz}{(1-(1-xy)z },
\end{equation} 
where $A_3,B_3 \in \Z$ are integers. The analysis appears in \cite{RV01}, and an expanded version of the theory of cellular integrals is presented in \cite[Section 5.3]{BF14}.\\

\begin{table}[h!]
\centering
\caption{Historical Data For $\mu(\zeta(3)$} \label{t6093.001}
\begin{tabular}{l|l|l}
Irrationality Measure Upper Bound&Reference&Year\\
\hline
 $\mu(\zeta(3) \leq 13.41782$ & Apery, \cite{AR79}&1979\\
 $\mu(\zeta(3) \leq 7.377956$ & Hata, \cite{HM00}&2000\\
 $\mu(\zeta(3) \leq  5.513891$ & Rhin, Viola, \cite{RV01}&2001\\
\end{tabular}
\end{table}
There some relationship between the numbers $\zeta(3)$ and $\pi^3$, but is not clear if $\mu(\zeta(3))=2$. In \cite{CN19}, it was proved that $\zeta(3)=\alpha\pi^3$, where $\alpha \in \R$ is irrational. The numerical data in Table \ref{t6093.09} suggests the followings. 

\begin{conj}\label{conj6093.20} The irrationanlity exponent of the first odd zeta constant is $\mu(\zeta(3))=\mu(\alpha \pi^3)=2$, where $\alpha \ne 0$ is a unique irrational number.
\end{conj}

%tttttttttttttttttttttttttttttttttttttttttttttttttttttttttttttttttttttttttttttttttttttttttttttttttttttttttttttttttttttttttttttttttttttttttttttttttttttttttttttttttttt
%AAAAAAAAAAAAAAAAAAAAAAAAAAAAAAAAAAAAAAAAAAAAAAAAAAAAAAAAAAAAAAAAAAAAAAAAAAAAAAA
\subsection{Numerical Data For The Exponent $\mu(\zeta(3))$}
The continued fraction of the first odd zeta constant is
\begin{equation}\label{eq6093.110}
\zeta(3)=[1, 2, 0, 2, 0, 5, 6, 9, 0, 3, 1, 5, 9, 5, 9, 4, 2, 8, 5, 3, 9, 9, 7, 3, 8, \ldots ],   
\end{equation}
listed as A002117 in OEIS. 
The sequence of convergents $\{p_n/q_n: n \geq 1\}$ is computed via the recursive formula provided in Lemma \ref{lem2000.101}.
The approximation $\mu_n(\zeta(3))$ of the exponent in the inequality
\begin{equation} \label{eq6094.46}
  \left | \pi^3-\frac{p_n}{q_n} \right | 
\geq\frac{1}{q^{\mu_n(\pi^3) }}
\end{equation}
are tabulated in Table \ref{t6093.09} for the early stage of the sequence of convergents $p_n/q_n \longrightarrow \pi^3$. 
%\newpage
\begin{table}[h!]
\centering
\caption{Numerical Data For The Exponent $\mu(\zeta(3))$} \label{t6093.09}
\begin{tabular}{l|l|l| l}
$n$&$p_n$&$q_n$&$\mu_n(\zeta(3))$\\
\hline
1&$1$&   $1$   &$ $\\
2&$5$&  $4$   &$2.191267 $\\
3&$6$&   $5$   &$3.843922$\\
4&$113$&  $94$   &$2.103378$\\
5&$119$&   $99$   &$2.222511$\\
6&$232$&  $193$   &$2.102718$\\
7&$351$&  $292$   &$2.302278$\\
8&$1636$&   $1361$   &$2.038931$\\
9&$1987$&  $1653$   &$2.309777$\\
10&$19519$&  $16238$   &$2.232018$\\ 
11&$177658$&      $147795$   &$2.084580$ \\ 
12&$374835$&   $311828$   &$2.057472$\\
13&$552493$&   $459623$   &$2.065833$\\
14&$927328$& $771451$ & $2.053480$ \\
15&$1479821$&  $1231074$ & $2.072380$ \\ 
16&$3886970$&  $3233599$ & $2.138006$ \\
17&$28688611$&    $23866267$ & $2.041149$ \\
18&$32575581$&  $27099866$ & $2.041133$ \\
19&$61264192$&  $50966133$ & $2.114414$ \\
20&$461424925$&  $383862797$ & $2.124760$ \\
21&$5136938367$&  $4273456900$ & $2.022499$ \\
22&$5598363292$&  $4657319697$ & $2.044823$ \\
23&$10735301659$&  $8930776597$ & $2.025155$ \\
24&$16333664951$&$13588096294$&$2.064764$\\
25&$59736296512$&  $49695065479$& $2.014150$\\
26&$76069961463$&   $63283161773$   &$2.082353$\\
27&$516156065290$&  $429394036117$   &$2.006174 $\\
28&$592226026753$&$492677197890$&$2.128367$\\
29&$18282936867880$&  $15209709972817$   &$2.007412$\\
30&$18875162894633$&  $15702387170707$ & $2.056200$\\
\end{tabular}
\end{table}

%sssssssssssssssssssssssssssssssssssssssssssssssssssssssssssssssssssssssssssssssssssssssssssssssssssssssssssssssssssssssssssssssssssssssssssssssssssssssssssssssssssssssssssssssssssssssssssssssssssssssssssssssssssssssssssssssssss
\section{The Exponent Result For $\pi^{k}$ } \label{s7729}
The method used to prove the irrationality measure $\mu(\pi^k)$ of the number $\pi^k$ is not based on rational functions approximations of $\pi^k$ and the prime number theorem. Some relevant references are \cite{RV96},  \cite{HM00}, \cite{MK53}, \cite{MM74}, \cite{CG82}, \cite{HM93}, \cite{SV08}, and \cite{BF00} for an introduction to the rational approximations of $\pi$ and the various proofs.\\

The proof is based on an effective upper bound of the reciprocal sine function over the sequence of $\{q_n: n\geq 1\}$ derived in Section \ref{s6699}. 
\begin{proof} (Theorem \ref{thm6010.02}) Let $\varepsilon >0$ be an arbitrary small number, and let $\{p_n/q_n: n \geq 1\}$ be the sequence of convergents of the irrational number $\pi^k$, with $k \geq 1$. By Theorem \ref{thm6699.300}, the reciprocal sine function has the upper bound
\begin{equation} \label{eq7729.883}
 \left |\frac{1}{\sin \left ( \pi^{k+1}q_n \right ) }\right |\ll q_n^{1+\varepsilon} .
\end{equation} 
Moreover, the relation $\sin \left ( \pi^{k+1}q_n \right )= \sin \left (\alpha p -\pi^{k+1} q_n \right )$ is true if and only if $\alpha p=\pi p_n$, where $p_n$ is an integer. These information lead to the following inequalities  
\begin{eqnarray} \label{eq7729.885}
\frac{1}{q_n^{1+\varepsilon}}&\ll&  \left |\sin \left (\pi^{k+1}q_n  \right )\right |\\
&\ll& \left |\sin \left (\pi^{k+1}q_n -\pi p_n \right )\right |\nonumber\\
&\ll& \left |\sin \left (\pi \left ( \pi^{k}q_n - p_n \right )\right )\right |\nonumber\\
&\ll& \left |\pi^{k} q_n- p_n \right |\nonumber
\end{eqnarray}
for all sufficiently large $p_n/q_n$. Therefore, 
\begin{eqnarray} \label{eq7729.36}
  \left | \pi^k-\frac{p_n}{q_n} \right | 
&\gg&\frac{1}{q^{2+\varepsilon} }\\
&=&\frac{1}{q^{\mu(\pi^{k})+\varepsilon} }\nonumber.
\end{eqnarray}
Clearly, this implies that the irrationality measure of the real number $\pi^k$ is $\mu(\pi^k)=2$, see Definition \ref{dfn2000.01}. 
Quod erat faciendum.
\end{proof}

%sssssssssssssssssssssssssssssssssssssssssssssssssssssssssssssssssssssssssssssssssssssssssssssssssssssssssssssssssssssssssssssssssssssssssssssssssssssssssssssssssssssssssssssssssssssssssssssssssssssssssssssssssssssssssssssssssss
\section{Basic Diophantine Approximations Results} \label{s2000}
All the materials covered in this section are standard results in the literature, see \cite{HW08}, \cite{LS95}, \cite{NZ91}, \cite{RH94}, \cite{SJ05}, \cite{WM00}, et alii. 

\begin{lem} \label{lem2000.101} Let $\alpha=\left [ a_0, a_1, \ldots, a_n, \ldots, \right ]$ be the continue fraction of the real number $\alpha \in \R$. Then the following properties hold.
%\begin{multicols}{2}
 \begin{enumerate} [font=\normalfont, label=(\roman*)]
\item$ \displaystyle  p_n=a_np_{n-1}+p_{n-2},$ \tabto{6cm} $p_{-2}=0, \quad p_{-1}=1$, \quad for all $n\geq 0.$
\item$ \displaystyle  q_n=a_nq_{n-1}+q_{n-2},$ \tabto{6cm} $q_{-2}=1, \quad q_{-1}=0$, \quad for all $n\geq 0.$
\item$ \displaystyle  p_nq_{n-1}-p_{n-1}q_{n}=(-1)^{n-1},$ \tabto{6cm} for all $n\geq 1.$
\item$ \displaystyle  \frac{p_n}{q_{n}}=a_0+\sum_{0 \leq k < n}\frac{(-1)^{k}}{q_kq_{k+1}},$ \tabto{6cm} for all $n\geq 1.$

\end{enumerate}
\end{lem}
%\end{multicols}

%qqqqqqqqqqqqqqqqqqqqqqqqqqqqqqqqqqqq
\subsection{Rationals And Irrationals Numbers Criteria} 
A real number \(\alpha \in \mathbb{R}\) is called \textit{rational} if \(\alpha = a/b\), where \(a, b \in \mathbb{Z}\) are integers. Otherwise, the number
is \textit{irrational}. The irrational numbers are further classified as \textit{algebraic} if \(\alpha\) is the root of an irreducible polynomial \(f(x) \in
\mathbb{Z}[x]\) of degree \(\deg (f)>1\), otherwise it is \textit{transcendental}.\\

\begin{lem} \label{lem2000.01} If a real number \(\alpha \in \mathbb{R}\) is a rational number, then there exists a constant \(c = c(\alpha )\) such that
\begin{equation}
\frac{c}{q}\leq \left|  \alpha -\frac{p}{q} \right|
\end{equation}
holds for any rational fraction \(p/q \neq \alpha\). Specifically, \(c \geq  1/b\text{ if }\alpha = a/b\).
\end{lem}

This is a statement about the lack of effective or good approximations for any arbitrary rational number \(\alpha \in \mathbb{Q}\) by other rational numbers. On the other hand, irrational numbers \(\alpha \in \mathbb{R}-\mathbb{Q}\) have effective approximations by rational numbers. If the complementary inequality \(\left|  \alpha -p/q \right| <c/q\) holds for infinitely many rational approximations \(p/q\), then it already shows that the real number \(\alpha \in \mathbb{R}\) is irrational, so it is sufficient to prove the irrationality of real numbers.

\begin{lem}[Dirichlet]\label{lem2000.02} 
 Suppose $\alpha \in \mathbb{R}$ is an irrational number. Then there exists an infinite
sequence of rational numbers $p_n/q_n$ satisfying
\begin{equation}
0 < \left|  \alpha -\frac{p_n}{q_n} \right|< \frac{1}{q_n^2}
\end{equation}
for all integers \(n\in \mathbb{N}\).
\end{lem}

\begin{lem} \label{lem2000.03}   Let $\alpha=[a_0,a_1,a_2, \ldots]$ be the continued fraction of a real number, and let $\{p_n/q_n: n \geq 1\}$ be the sequence of convergents. Then
\begin{equation}
0 < \left|  \alpha -\frac{p_n}{q_n} \right|< \frac{1}{a_{n+1}q_n^2}
\end{equation}
for all integers \(n\in \mathbb{N}\).
\end{lem}
This is standard in the literature, the proof appears in \cite[Theorem 171]{HW08}, \cite[Corollary 3.7]{SJ05}, \cite[Theorem 9]{KA97}, and similar references.\\

\begin{lem} \label{lem2000.05}   Let $\alpha=[a_0,a_1,a_2, \ldots]$ be the continued fraction of a real number, and let $\{p_n/q_n: n \geq 1\}$ be the sequence of convergents. Then
\begin{multicols}{2}
 \begin{enumerate} [font=\normalfont, label=(\roman*)]
\item$ \displaystyle 
 \frac{1}{2q_{n+1}q_n} \leq \left | \alpha - \frac{p_n}{q_n}  \right | \leq \frac{1}{q_n^{2}} 
$,
\item$\displaystyle
\frac{1}{2a_{n+1}q_n^2} \leq \left | \alpha - \frac{p_n}{q_n}  \right | \leq \frac{1}{q_n^{2}} 
$,
\end{enumerate}
\end{multicols}
for all integers \(n\in \mathbb{N}\).
\end{lem}
The recursive relation $q_{n+1}=a_{n+1}q_n+q_{n-1}$ links the two inequalities. Confer \cite[Theorem 3.8]{OC63}, \cite[Theorems 9 and 13]{KA97}, et alia. The proof of the best rational approximation stated below, appears in \cite[Theorem 2.1]{RH94}, and \cite[Theorem 3.8]{SJ05}. 
\begin{lem} \label{lem2000.07}   Let $\alpha \in \R$ be an irrational real number, and let $\{p_n/q_n: n \geq 1\}$ be the sequence of convergents. Then, for any rational number $p/q \in \Q^{\times}$, 
\begin{multicols}{2}
 \begin{enumerate} [font=\normalfont, label=(\roman*)]
\item$ \displaystyle 
 \left | \alpha q_n - p_n  \right | \leq  \left | \alpha q -p  \right |
$,
\item$\displaystyle
 \left | \alpha - \frac{p_n}{q_n}  \right | \leq  \left | \alpha - \frac{p}{q}  \right |
$,
\end{enumerate}
\end{multicols}

for all sufficiently large \(n\in \mathbb{N}\) such that $q \leq q_n$.
\end{lem}

%subsubqqqqqqqqqqqqqqqqqqqqqqqqqqqqq
\subsection{ Irrationalities Measures }
 
The concept of measures of irrationality of real numbers is discussed in \cite[p.\ 556]{WM00}, \cite[Chapter 11]{BB87}, et alii. This concept can be approached from several points of views. 

\begin{dfn} \label{dfn2000.01} {\normalfont The irrationality measure $\mu(\alpha)$ of a real number $\alpha \in \R$ is the infimum of the subset of  real numbers $\mu(\alpha)\geq1$ for which the Diophantine inequality
\begin{equation} \label{eq597.36}
  \left | \alpha-\frac{p}{q} \right | \ll\frac{1}{q^{\mu(\alpha)} }
\end{equation}
has finitely many rational solutions $p$ and $q$. Equivalently, for any arbitrary small number $\varepsilon >0$
\begin{equation} \label{eq597.36}
  \left | \alpha-\frac{p}{q} \right | \gg\frac{1}{q^{\mu(\alpha)+\varepsilon} }
\end{equation}
for all large $q \geq 1$.
}
\end{dfn}
\begin{thm} \label{thm2000.33} {\normalfont  (\cite[Theorem 2]{BY08}) } The map $\mu : \mathbb{R} \longrightarrow [2,\infty) \cup \{1\}$ is surjective function. Any number in the set $[2, \infty) \cup \{1\}$ is the irrationality measure  of some irrational number.
\end{thm} 

\begin{exa} \label{ex2000.33} {\normalfont Some irrational numbers of various  irrationality measures.

\begin{enumerate} [font=\normalfont, label=(\arabic*)]
\item  A rational number has an irrationality measure of $\mu(\alpha)=1$, see \cite[Theorem 186]{HW08}.
\item   An algebraic irrational number has an irrationality measure of $\mu(\alpha)=2$, an introduction to the earlier proofs of  Roth Theorem appears in \cite[p.\ 147]{RH94}.
\item   Any irrational number has an irrationality measure of $\mu(\alpha)\geq 2$.
\item   A Champernowne number $\kappa_b=0.123 \cdots b-1\cdot b \cdot b+1 \cdot b+2\cdots$ in base $b\geq 2$, concatenation of the $b$-base integers, has an irrationality measure of $\mu(\kappa_b)=b$. For example, the decimal number
\begin{equation}
\kappa_{10}=0.1234567891011121314151617\cdots
\end{equation}
has the irrationality measure of $\mu(\kappa_{10})=10$.
\item   A Mahler number $\psi_b=\sum_{n \geq 1} b^{-[\tau]^ n}$ in base $b\geq 3$ has an irrationality measure of $\mu(\psi_b)=\tau$, for any real number $\tau \geq 2$, see \cite[Theorem 2]{BY08}. For example, the decimal number
\begin{equation}
\psi_{10}=\frac{1}{10^{3}}+\frac{1}{10^{9}}+\frac{1}{10^{27}}+\frac{1}{10^{81}}+\cdots
\end{equation}
has the irrationality measure of $\mu(\psi_{10})=3$.
\item   A Liouville number $\ell_b=\sum_{n \geq 1} b^{-n!}$ parameterized by $b \geq 2$ has an irrationality measure of $\mu(\ell_b)=\infty$, see \cite[p.\ 208]{HW08}. For example, the decimal number
\begin{equation}
\ell_{10}=\frac{1}{10}+\frac{1}{10^{2}}+\frac{1}{10^{6}}+\frac{1}{10^{24}}+\cdots
\end{equation}
has the irrationality measure of $\mu(\ell_{10})=\infty$.

\end{enumerate}
}
\end{exa}
\begin{dfn} \label{dfn2000.03} {\normalfont A measure of irrationality $\mu(\alpha)\geq 2 $ of an irrational real number $\alpha \in \R^{\times}$ is a map $\psi:\N \rightarrow \R$ such that for any $p,q \in \N$ with $q\geq q_0$, 
\begin{equation} \label{eq2000.70}
\left | \alpha - \frac{p}{q}  \right | \geq \frac{1}{\psi(q)} .
\end{equation}
Furthermore, any measure of irrationality of an irrational real number satisfies $\psi(q) \geq \sqrt{5}q^{\mu(\alpha)}\geq \sqrt{5}q^2$.

}
\end{dfn}
\begin{thm} \label{thm2000.03} For all integers $p,q\in \N$, and $q \geq q_0$, the number $\pi$ satisfies the rational approximation inequality 
\begin{equation} \label{eq2000.75}
\left | \pi - \frac{p}{q}  \right | \geq \frac{1}{q^{7.6063}} .
\end{equation}
\end{thm}
\begin{proof} Consult the original source \cite[Theorem 1]{SV08}. 
\end{proof}

%BBBBBBBBBBBBBBBBBBBBBBBBBBBBBBBBBBBBBBBBBBBBBBBBBBBBBBBBBBBBBBBBBBBBBBBBBBBBBBBBBBBBBBBBBBBBBBBBBBBBBBBBBBBBBBBBBBBBBBBBBBBBBBBBBBBBBBBBBB

\currfilename.\\

\end{document}